\documentclass[12pt]{article} 
\usepackage{amsthm,amssymb,amsmath}
\usepackage{enumerate}
\usepackage{graphicx}
\usepackage{float}
\usepackage{multicol}

\usepackage{mathpazo}
\usepackage{subfig}
\usepackage{tikz}
\usepackage{color}
\definecolor{darkblue}{rgb}{0,0,.4}
\usepackage[bookmarks,
colorlinks=true,
linkcolor=darkblue,
citecolor=darkblue,
urlcolor=darkblue]{hyperref}
\hypersetup{pdfauthor=Paul M. Rakotomamonjy,
	pdfsubject=Combinatorics,
	pdftitle=Expected Patterns}

\theoremstyle{plain}
\newtheorem{theorem}{Theorem}
\newtheorem{corollary}[theorem]{Corollary}
\newtheorem{lemma}[theorem]{Lemma}
\newtheorem{proposition}[theorem]{Proposition}

\theoremstyle{definition}

\newtheorem{conjecture}[theorem]{Conjecture}

\newcommand{\rci}{\rm rci}
\newcommand{\rc}{\rm rc}

\newcommand{\exc}{{\rm exc}}
\newcommand{\inv}{{\rm inv}}
\newcommand{\des}{{\rm des}}

\newcommand{\crs}{{\rm cr}}
\newcommand{\nes}{{\rm nes}}

\newcommand{\ut}{{\rm ut}}
\newcommand{\lt}{{\rm lt}}

\newcommand{\T}{{\rm T}}

\usepackage{fullpage}

\numberwithin{equation}{section} 
\numberwithin{theorem}{section} 
\usepackage{vmargin}
\setmarginsrb{2.5cm}{1.5cm}{2.5cm}{2cm}{1.5cm}{0.5cm}{1cm}{0.5cm}
\usepackage{fancyhdr}
\title{New combinatorial interpretations of the binomial coefficients}

\author{Paul M. Rakotomamonjy$^1$ and Sandrataniaina R. Andriantsoa$^2$ \\
	\small Department of Mathematics and Computer Science\\[-0.8ex]
	\small Sciences and Technology, PB 906 Antananarivo 101, Madagascar\\[-0.8ex]
	\small \texttt{$^1$rpaulmazoto@gmail.com, $^2$andrian.2sandra@gmail.com}}
\date{}

\begin{document} 
	\maketitle
	\begin{abstract} 
		Using generating functions and some trivial bijections, we show in this paper that the binomial coefficients  count the set of (123,132) and (123,213)-avoiding permutations according to the number of crossings. We also define a q-tableau of power of two and prove that it counts the set of (213,312) and (132,312)-avoiding permutations according to the number of crossings.
		\begin{center}
			\textbf{Keywords:} Binomial coefficients, number of crossings,  restricted permutations, central polygonal numbers, q-tableau of power of two.
			
			\textbf{2010 Mathematics Subject Classification}: 05A20	 and 05A05.
		\end{center}
	\end{abstract}

	
	\section{Introduction and main results}\label{sec1}
	A \textit{statistic} st on a given set $E$ is a map from E to $\mathbb{N}$. We denote by $\sum_{e\in E} x^{{\rm st}(e)}$ the polynomial distribution of the statistic st over the set E. 
	Let $n$ be an integer such that $n\geq 1$. A permutation $\sigma$ of $[n]$ is a bijection from $[n]$ to itself which can be written linearly as $\sigma=\sigma(1)\sigma(2)\ldots \sigma(n)$. We denote by $S_n$ the set of all permutations of $[n]$. 		
	Recall that a \textit{crossing}  in a permutation $\sigma$ is a pair of indexes $(i,j)$ such that $i<j< \sigma(i)<\sigma(j)$  or  $\sigma(i)<\sigma(j)\leq i<j$. A \textit{nesting} of $\sigma$ is similarly defined as a pair $(i,j)$ such that $i<j< \sigma(j)<\sigma(i)$  or  $\sigma(j)<\sigma(i)\leq i<j$. We denote respectively by $\crs(\sigma)$ and $\nes(\sigma)$ the number of crossings and nestings of $\sigma$. We draw arc diagrams of a given permutation $\sigma$ in the figure bellow for illustration of crossings.
	
	\begin{figure}[h]
		\begin{center}
			\begin{tikzpicture}
			\draw[black] (0,1) node {$1\ 2\ 3\ 4\ 5\ 6\ 7$}; 
			\draw (-0.9,1.2) parabola[parabola height=0.2cm,red] (-0,1.2);\draw[->,black] (-0.09,1.25)--(-0,1.2);
			
			\draw (-0.6,1.2) parabola[parabola height=0.3cm,red] (0.85,1.2); \draw[->,black] (0.8,1.25)--(0.85,1.2);
			
			\draw (-0.45,0.7) -- (-0.35,0.85)[rounded corners=0.1cm] -- (-0.25,0.7) -- cycle; 
			
			\draw (-0.03,1.2) parabola[parabola height=0.2cm,red] (0.3,1.2); \draw[->,black] (0.27,1.25)--(0.3,1.2);
			
			\draw (-0.9,0.8) parabola[parabola height=-0.3cm,red] (0.3,0.8); \draw[->,black] (-0.85,0.75)--(-0.9,0.8);
			
			\draw (-0.6,0.8) parabola[parabola height=-0.3cm,red] (0.6,0.8); \draw[->,black] (-0.55,0.75)--(-0.6,0.8);
			
			\draw (0.6,0.8) parabola[parabola height=-0.1cm,red] (0.9,0.8); \draw[->,black] (0.65,0.75)--(0.6,0.8);			
			\end{tikzpicture}
			\caption{Arc diagrams of $\sigma=4735126 \in S_7$ with $\crs(\sigma)=3$ and $\nes(\sigma)=3$.}
			\label{fig:arcdiagproof}
		\end{center}
	\end{figure}
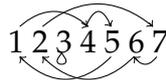

	Let $\sigma\in S_n$ and $\tau \in S_k$ with $1\leq k\leq n$. For a given increasing sequence of integers $i_1<i_2<\ldots <i_k$, we say that the subsequence $s=\sigma(i_1)\sigma(i_2)\ldots \sigma(i_k)$ of $\sigma$ is an occurrence of $\tau$ if $s$ and $\tau$ are in order isomorphic, i.e $\sigma(i_x)<\sigma(i_y)$ if only if $\tau(x)<\tau(y)$. If there is no occurrence of the pattern $\tau$ in $\sigma$, we also say that $\sigma$ is $\tau$-\textit{avoiding}. For example, the permutation $\pi=4162375 \in S_7$ has five occurrences  of 312. These are 312, 423, 623, 625,  and 635. It is obvious to see that $\pi$ is $321$-avoiding. There are three useful trivial involutions on $S_n$ namely \textit{reverse}, \textit{complement} and \textit{inverse} that simplify the search for equivalence classes. For any $\sigma=\sigma(1)\sigma(2)\ldots \sigma(n) \in S_n$,
	\begin{itemize}
		\item the reverse of $\sigma$ is r$(\sigma)=\sigma(n)\sigma(n-1)\ldots \sigma(1)$,
		\item the complement of $\sigma$ is c$(\sigma)=(n+1-\sigma(1))(n+1-\sigma(2))\ldots (n+1-\sigma(n))$,
		\item the inverse of $\sigma$ is i$(\sigma)=p(1) p(2) \ldots p(n)$ where $p(i)$ is the position of $i$ in $\sigma$. We often write i$(\sigma)=\sigma^{-1}$.
	\end{itemize}	
	By composition $\circ$, these involutions generate the dihedral group $\mathcal{D}=\{{\rm id,r,c,i,rc,ri,ci,rci}\}$ where ${\rm fg=f\circ g}$ for any f and g in $\{{\rm r,c,i}\}$. Let us consider $\pi=4135762 \in S_7$. We have r$(\pi)=2675314$, c$(\pi)=4753126$,  $\pi^{-1}=2731465$, rc$(\pi)=6213574$ and rci$(\pi)=3247516$.  We will denote by $S_n(\tau)$ the set of all $\tau$-avoiding permutations of $[n]$.
	For any given subset of patterns $\T$, we write $S_n(\T):=\cap_{\tau \in T}S_n(\tau)$ and $S(\T):=\cup_{n\geq 0}S_n(\T)$.  We say that $\T$ and $\T'$ are Wilf-equivalent if  $|S_n(\T)|=|S_n(\T')|$. 
	For example, we have $|S_n(123)|=|S_n(321)|$ and  $|S_n(312)|=|S_n(231)| = |S_n(132)|=|S_n(213)|$ since 321=r(123), 312=231$^{-1}$, 132=r(231) and 213=r(312).
	In fact, it is well known in  \cite{Knuth} that there is only one Wilf-equivalence class in $S_3$ and we have  $|S_n(\tau)|=\frac{1}{n+1}\binom{2n}{n}$ (the n-th Catalan numbers) for any pattern $\tau\in S_3$. Bijective proof of identity $|S_n(132)|=|S_n(321)|$ interested several author, eg. see \cite{ClaKitaev} and reference therein. It was also shown by Simion and Schmidt \cite{SiS} that there are three Wilf-equivalence classes for two patterns of $S_3$. 
	
	For any given statistic st, the notion of st-wilf-equivalence generalizes that of Wilf-equivalence. Say that two subset of patterns $\T_1$ and $\T_2$ are st-Wilf-equivalent and we write $\T_1=_{{\rm st}}\T_2$ if only if 
	$\sum_{\sigma \in S_n(\T_1)} x^{{\rm st}(\sigma)}=\sum_{\sigma\in S_n(\T_2)} x^{{\rm st}(\sigma)}$. For this notion,  we are much inspired by the work of some author. In \cite{ARob}, Robertson et al. enumerated the set of permutations that avoid a single pattern of length 3 according to the number of fixed points. Instead of considering simply the number of fixed points,  Elizalde 
	\cite{Eliz1,Eliz2} extended the result of Robertson et al. and studied its joint  distribution with the number of excedances. He also treated all pairs and triples patterns of length 3. Dokos et al. \cite{Dokos} provided the number of inversions and major index-Wilf equivalence classes for singleton, double , triple and four patterns of length 3. In this paper, we hope to extend the recent result of the first author \cite{Rakot} of this paper who introduced the study of the number of crossings and nestings on permutations that avoid a single pattern of length 3. He proved bijectively that the three patterns 321, 132 and 213 are \crs-Wilf-equivalent, i.e.
	\begin{equation}\label{eq:res1}
	\sum_{\sigma\in S_n(321)} q^{\crs(\sigma)}=\sum_{\sigma\in S_n(132)} q^{\crs(\sigma)}=\sum_{\sigma\in S_n(213)} q^{\crs(\sigma)}.
	\end{equation}
	To prove the first identity of (\ref{eq:res1}), he exploited the bijection $\Theta:S_n(321)\rightarrow S_n(132)$ of Elizalde and Pak exhibited in \cite{ElizP} and proved that it is \crs-preserving. The second identity of (\ref{eq:res1}) is simply obtained from the fact that the bijection rci preserves also the statistic $\crs$. Notice that the bijection $\Gamma:S_n(321)\rightarrow S_n(132)$ defined by Robertson \cite{ARob2} is also \crs-preserving. Indeed, Saracino discovered recently after his interesting joint work with Bloom \cite{Bloom, Bloom2} that the bijections $\Theta$ and $\Gamma$ are related each other by the simple relation $\Gamma=\Theta\circ\rci$.\\ 
	Using the q,p-Catalan numbers of Randrianarivony \cite{ARandr}, Rakotomamonjy expressed in terms of continued fraction the generating function of $\displaystyle \sum_{\sigma\in S_n(\tau)} q^{\crs(\sigma)}$ for  any $\tau \in \{321,132,213\}$, 
	\begin{equation*}
	\sum_{\sigma\in S(\tau)} q^{\crs(\sigma)}z^{|\sigma|}=\frac{1}{1-\displaystyle\frac{z}{				
			1-\displaystyle\frac{z}{
				1-\displaystyle\frac{qz}{								
					1-\displaystyle\frac{qz}{
						1-\displaystyle\frac{q^2z}{
							1-\displaystyle\frac{q^2z}{				
								\ddots}
						}
				}}		
	}}}. \label{eq:fc}
	\end{equation*}
	Finding $\sum_{\sigma\in S(\tau)} q^{\crs(\sigma)}z^{|\sigma|}$ is staying unsolved for any pattern  $\tau$ in $\{231,312,123\}$. We observe that this continued fraction  appeared in \cite{BEliz} as the distribution of occurrences of generalized patterns in 231-avoiding permutations.We discuss this here for interest readers on the correspondence between these results. View that Corteel \cite{Cort} established the connection between occurrences of patterns, crossings and nestings on permutations.
	
	The first purpose of this paper  is about enumeration result  of (123,132) and (123,213)-avoiding permutations according to the number of crossings. We discover that the binomial coefficients involve on enumeration.  For that, we denote by $S_n^k=\{\sigma \in S_n/ \sigma(n+1-k)=1\}$ the set of permutations of $[n]$ where the position of 1 is $n+1-k$, $1\leq k\leq n$. Moreover, we also denote by $S_{n,k}=\{\sigma \in S_n/ \sigma(n)=k\}$ the set of permutations of $[n]$ ending with $k$. As example, $S_n^1=S_{n,1}=\{\sigma \in S_n/ \sigma(n)=1\}$ is the set of permutations ending with 1.
	\begin{theorem} \label{main1} For all integer $n\geq 2$, we have   
		\begin{equation*}
		\sum_{\sigma \in S_{n}^2(123,132)}q^{\crs(\sigma)} =\sum_{\sigma \in S_{n,2}(123,213)}q^{\crs(\sigma)}= (1+q)^{n-2}.
		\end{equation*}
\end{theorem}
	\hspace{-0.6cm}To prove this theorem, we first study the fine structure of $S_{n}(123,132)$ in order to find a suitable partition which leads to a recurrence relation for the corresponding polynomial distribution of number of crossings. Notice that similar results as theorem  \ref{main1} on enumeration of restricted permutations refined by number of descents and inversions are found respectively in \cite{Bala,BBES} and \cite{Chung,Dokos}.
	
	For the second purpose of this work, let us consider the q-tableau  $(R_{n}^{k}(q))_{n,k}$ of powers of two defined by the following way
	\begin{equation} \label{q-tableau}	
	\begin{cases}
	R_{n}^n(q)=R_{n}^{n-1}(q)=1 & \\
	R_{n}^{k}(q)=q^{\min\{k-1,n-1-k\}}R_{n-1}^{k}(q)+R_{n}^{k+1}(q) & \text{ if \ \ $0<k<n-1$} \\		
	R_{n}^0(q)=R_{n-1}^{0}(q)+R_{n}^{1}(q) & 
	\end{cases}
	\end{equation}
	It is not difficult to see that $R_{n}^{k}(1)=2^{n-1-k}$ for $0\leq k< n$. Hence, we have $\sum_{k=0}^{n}R_n^k(1)=2^{n}$.  We present here some values for $(R_{n}^k(q))$ in table \ref{table:qtable}.
	\begin{center}
		\begin{table}[h]
			\begin{tabular}{|c|llll|}
				\hline 
				k& 0&1 & 2 & 3  \\
				\hline
				0& $1$&&  &    \\
				1& $1$ &$1$&  &    \\
				2&$2$ &$1$& $1$&  \\
				3& $4$ &$2$& $1$& $1$ \\
				4& $7+q$ &$3+q$& $1+q$&  $1$ \\
				5& $11+4q+q^2$&$4+3q+q^2$& $1+2q+q^2$& $1+q$ \\
				6&$16+9q+5q^2+2q^3$ &$5+5q+4q^2+2q^3$& $1+2q+3q^2+2q^3$& $1+q+q^2+q^3$\\
				\hline 
			\end{tabular}
			\captionof{table}{ Few values of $(R_{n}^k(q))$ for $0\leq n\leq 6$ and $0\leq k\leq 3$} 
			\label{table:qtable} 
		\end{table}
	\end{center}
	If we denote by $S_n^{[k]}=\{\sigma \in S_n/\sigma(n+1-i)=i \text{ for all } i \in [k] \}$ for $1\leq k \leq n$, we have the following result which means that this defined q-tableau counts the set of permutations that avoid some pairs of patterns according to the number of crossings. We obtain it from certain manipulation of the fine structure of $S_n(213,312)$.
	\begin{theorem}\label{main2}	
		Let $\T$ be one of the pairs $\{213,312\}$ and $\{132,312\}$. For all non-negative integer $n$, we have the following identities
		\begin{equation*}\label{eq-main1}
		\sum_{\sigma \in S_n(\T)}q^{\crs(\sigma)}=R_{n}^0(q) \text{ and } 	\sum_{\sigma \in S_n^{[k]}(\T)}q^{\crs(\sigma)}=R_{n}^k(q) \text{ for all } k\geq 1.
		\end{equation*}
	\end{theorem}

	So, we organize the rest of this paper as follow. In section \ref{sec2}, we introduce some notations and define some trivial bijections on $S_n$. We also try to use the defined bijections to get a simple relationship between the distribution of the number of crossings over the set $S_n(231)$ and $S_n(312)$. In section \ref{sec3}, we use some tools from the previous one and  establish the proof of theorem \ref{main1}. In section \ref{sec4}, we provide the proof of theorem \ref{main2} and extend it for the polynomial distribution of the number of crossings over the set $S_n(213, 231)$ and $S_n(132, 231)$. We end this paper with enumeration of (321,231)-avoiding permutations according to the number of excedances and crossings  that links the results of Chung et al. and Dokos et al on the distribution of the number of inversions  and descents.

	\section{Trivial bijections and first uses}\label{sec2}
	In this section, we present two trivial bijections that we need in the next sections to prove some identities. Using the following well known property of \cite{SiS} 
	\begin{equation*}
	f(S_n(\T))= S_n(f(\T)) \text{ for any  $f \in \mathcal{D}=\{{\rm id,r,c,i,rc,ri,ci,rci}\}$ and subset of patterns \T},
	\end{equation*}
	we can use our bijections to enumerate avoiding permutations according to the number of crossings. We hope to extend the result of Rakotomamonjy who showed that the trivial involution $\rci$ on $S_n$ is $\crs$-preserving and if $\T'= \rci(\T)$ we have consequently
\begin{equation*}
	\sum_{\sigma \in S_n(\T)} q^{\crs(\sigma)}=\sum_{\sigma \in S_n(\T')} q^{\crs(\sigma)}.
\end{equation*}
We let the reader to show  that we have $S_{n,k}$ = \rci$(S_n^k)$ for all integers $n$ and $k\geq 1$. So, if $\T'= \rci(\T)$, we also have
\begin{equation}\label{rel21}
\sum_{\sigma \in S_n^k(\T)} q^{\crs(\sigma)}=\sum_{\sigma \in S_{n,k}(\T')} q^{\crs(\sigma)}.
\end{equation}
	let us recall some needed notations that Rakotomamonjy has used in \cite{Rakot}. Given a permutation $\sigma$  and two integers $a$ and $b$, we denote by $\sigma^{(a,b)}$ the obtained permutation from $\sigma$ by the following way:
	\begin{itemize}
		\item add by $1$ each number in $\sigma$ which is greater or equal to $b$,
		\item then, insert $b$ at the $a$-th position of the modified $\sigma$.
	\end{itemize}
	To simplify, we write $\sigma^{-(a,b)}$ for $(\sigma^{-1})^{(a,b)}$.  As example, we have $3142^{(2,\textcolor{gray}{3})}=4\textcolor{gray}{3}152$ and $3142^{-(2,\textcolor{gray}{3})}=2\textcolor{gray}{3}514$. If $\sigma \in S_{n}$, we have particularly
	$\sigma^{(n+1,1)}(i)=\sigma(i)+1$  if $i\leq n$ and $\sigma^{(n+1,1)}(n+1)=1$. Let also denote respectively by $$\ut(\sigma):=|\{i/\sigma^{-1}(i)<i<\sigma(i)\}|  \text{ and }	\lt(\sigma):=|\{i/\sigma(i)<i<\sigma^{-1}(i)\}|$$ the number of upper and lower transients of a given permutation $\sigma$. Notice that lower transients of a given permutation $\sigma$ are counted as crossing of $\sigma$. The following lemma even comes from definitions of $\sigma^{(n+1,1)}$.
	\begin{lemma}\label{lem21} 	For any $\sigma \in S_{n}$, we have $\crs(\sigma^{(n+1,1)})=\crs(\sigma)+\ut(\sigma)-\lt(\sigma).$	
	\end{lemma}
	\begin{proof}
		Let $\sigma \in S_{n}$ and $\pi=\sigma^{(n+1,1)}$.  From definition  of $\pi$, the following statements hold.
		\begin{enumerate}
			\item $i< j< \sigma(i)<\sigma(j)$ if only if  $i< j< j+1<\pi(i)<\pi(j)$.
			\item $\sigma(i)<\sigma(j)<i< j$ if only if  $\pi(i)<\pi(j)\leq i< j$.
			\item $\sigma^{-1}(i)< i< \sigma(i)$ if only if  $\pi^{-1}(i+1)< i$ and  $\pi(i)>i+1$. It means that $(\sigma^{-1}(i),i)$ which does not a crossing of $\sigma$ becomes one of $\pi$.
			\item $\sigma(i)< i< \sigma^{-1}(i)$ if only if  $\pi(i) \leq i$ and  $i+1\leq\pi^{-1}(i)$. It means that  $(i,\sigma^{-1}(i))$ is counted as a crossing of $\sigma$ becomes no longer one of $\pi$. 	 
		\end{enumerate}
		Graphical illustration of the statements 3. an 4. are given in the figure \ref{fig:arcdiag3} below. From these four statements, we get $\crs(\pi)=\crs(\sigma)+|\{i/\sigma^{-1}(i)< i< \sigma(i)\}|-|\{i/\sigma(i)< i< \sigma^{-1}(i)\}|=\crs(\sigma)+\ut(\sigma)-\lt(\sigma)$. 
	\end{proof}
	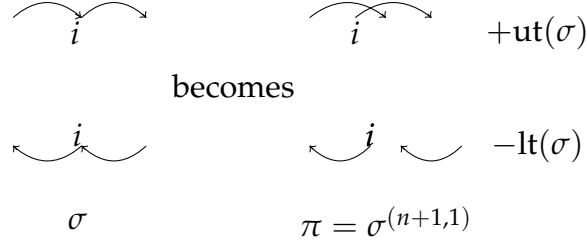
\begin{figure}[h]
		\begin{center}
			\begin{tikzpicture}
			\draw (-0.9,1.2) parabola[parabola height=0.2cm,red] (-0,1.2); \draw[->,black] (-0.085,1.25)--(-0,1.2);
			\draw (0,1.2) parabola[parabola height=0.2cm,red] (0.85,1.2);  \draw[->,black] (0.8,1.25)--(0.85,1.2);
			\draw[black] (-0.08,1) node {$i$}; \draw[black] (3.6,1) node {$i$};	 \draw[black] (6,1) node {$+\ut(\sigma)$};	
			
			\draw (3,1.2) parabola[parabola height=0.2cm,red] (4,1.2); 
			\draw[->,black] (3.95,1.25)--(4,1.2);
			\draw (3.6,1.2) parabola[parabola height=0.2cm,red] (4.6,1.2); 	\draw[->,black] (4.55,1.25)--(4.6,1.2);	
			
			\draw[black] (2,0.3) node {becomes};
			
			\draw (-0.9,-0.5) parabola[parabola height=-0.2cm,red] (-0,-0.5); \draw[->,black] (-0.85,-0.55)--(-0.9,-0.5);	
			\draw (0,-0.5) parabola[parabola height=-0.2cm,red] (0.85,-0.5); \draw[->,black] (0.05,-0.55)--(-0,-0.5);	
			
			\draw (3,-0.5) parabola[parabola height=-0.2cm,red] (3.8,-0.5); \draw[->,black] (3.05,-0.55)--(3,-0.5);
			\draw (4.2,-0.5) parabola[parabola height=-0.2cm,red] (5,-0.5); 
			\draw[->,black] (4.25,-0.55)--(4.2,-0.5);
			\draw[black] (-0.05,-0.35) node {$i$}; \draw[black] (3.8,-0.35) node {$i$};	
			
			\draw[black] (-0.05,-1.5) node {$\sigma$}; \draw[black] (4,-1.5) node { $\pi=\sigma^{(n+1,1)}$};
			
			\draw[black] (3.8,-0.35) node {$i$}; \draw[black] (6,-0.5) node {$-\lt(\sigma)$};	
			
			\end{tikzpicture}
			\caption{Graphical illustration of the statements 3. and 4. in the proof of lemma \ref{lem21}.}
			\label{fig:arcdiag3}
		\end{center}
	\end{figure}
	
	\hspace{-0.6cm }We can similarly prove that we the following lemma holds.  
	\begin{lemma}\label{lem23}
		For any $\sigma \in S_n$, we have $\crs(\sigma^{(n,1)})=\crs(\sigma)+1-\delta_{n,\sigma(n)}+\ut(\sigma)-\lt(\sigma)$.
	\end{lemma}
Before presenting our bijections, we need the following lemma which completes the previous ones.
\begin{lemma}\label{lem24}
	Let $\sigma$ be a given permutation and $\pi=\sigma^{-1}$ or ${\rm rc}(\sigma)$, we have $$\crs(\pi)=\crs(\sigma)+\ut(\sigma)-\lt(\sigma).$$ 
\end{lemma}
	\begin{proof}
		Let $\sigma \in S_n$ and $\pi=\sigma^{-1}$ or ${\rm rc}(\sigma)$. Applying reverse-complement or inverse  on $\sigma$ exchanges lower and upper arcs including of course transients, i.e. $\ut(\pi)=\lt(\sigma)$ and $\lt(\pi)=\ut(\sigma)$. On the first hand, each upper transient of $\sigma$, which is not counted as crossings of $\sigma$,  becomes one of $\pi$.  On the second hand, each lower transient of $\sigma$, which is counted as crossing of $\sigma$,  becomes no longer one of $\pi$. This explain how we get $\crs(\pi)=\crs(\sigma)+\ut(\sigma)-\lt(\sigma)$.
	\end{proof}
	Let us consider the  maps $\phi_k$ and $\psi_k$ from $S_n$ to $S_{n+1}^{k}$  with $k\geq 1$ that are defined respectively as follow,
$$\phi_k: \sigma \mapsto \sigma^{-(n+2-k,1)} \text{ and } \psi_k: \sigma \mapsto {\rm rc}(\sigma)^{(n+2-k,1)}.$$
As example, we have $\phi_3(31542)=361254$ and $\psi_3(31542)=531264$.	
	It is clear that the maps $\phi_k$ and $\psi_k$ are bijective for any $k\geq 1$. Furthermore,  we can use the three previous lemmas to prove the following proposition.
\begin{proposition}\label{prop1} The bijections $\phi_1$ and $\psi_1$ are $\crs$-preserving and  the bijection $\phi_2$ satisfies $\crs(\phi_2(\sigma))=\crs(\sigma)+1-\delta_{n,\sigma(n)}$ for any $\sigma \in S_n$.
\end{proposition}
\begin{proof}
	Combining lemmas \ref{lem21} and \ref{lem23}, we get $\crs(\sigma^{-(n+1,1)})=\crs({\rm rc}(\sigma)^{(n+1,1)})=\crs(\sigma)$ for any $\sigma \in S_n$. This implies the cr-preserving of $\phi_1$ and $\psi_1$. By the same way,  when we combine lemmas \ref{lem21} and \ref{lem24}, we get $\crs(\phi_2(\sigma))=\crs(\sigma)+1-\delta_{n,\sigma(n)}$ for any $\sigma \in S_n$.
\end{proof}
	For any subset of permutations T  and any integer n, we denote by $F_n(\T;q)=\sum_{\sigma \in S_{n}(\T)} q^{\crs(\sigma)}$ the polynomial distribution of the number of crossings over the set of T-avoiding permutations and $F(\T;q,z)$ its generating function, i.e. 
	 $$F(\T;q,z)=\sum_{\sigma \in S(\T)} q^{\crs(\sigma)}z^{|\sigma|}=1+ \sum_{n\geq 1}F_n(\T;q) z^n.$$ 
	Particularly, we set $F_n(q)=\sum_{\sigma \in S_{n}} q^{\crs(\sigma)}=F_n(\emptyset;q)$.
	
\begin{theorem} \label{thm23}For all integer $n\geq 1$, we have
		$$\sum_{\sigma \in S_{n+1}^1} q^{\crs(\sigma)}=F_n(q) \text{ and } \sum_{\sigma \in S_{n+1}^2} q^{\crs(\sigma)}=qF_n(q)+(1-q)F_{n-1}(q).$$
\end{theorem}
\begin{proof}
	The first identity of theorem \ref{thm23} comes from the cr-preserving of the bijection $\phi_1$ (or $\psi_1$). The second one obviously use the bijection $\phi_2$ and its property
\begin{equation*}
	\crs(\phi_2(\sigma))=	
	\begin{cases}
		\crs(\sigma)  &    \text{ if } \ \ \sigma(n)=n \\
		\crs(\sigma)+1  &    \text{ if } \ \ \sigma(n)<n
	\end{cases}.
\end{equation*} 
So, we get
\begin{eqnarray*}
\sum_{\phi_2(\sigma) \in S_{n+1}^2} q^{\crs(\phi_2(\sigma))}&=&q\times \sum_{\sigma \in S_{n},\sigma(n)\neq n} q^{\crs(\sigma)}+\sum_{\sigma \in S_{n},\sigma(n)= n} q^{\crs(\sigma)}.
\end{eqnarray*}	
Since 	$\displaystyle \sum_{\sigma \in S_{n},\sigma(n)\neq n} q^{\crs(\sigma)}=\sum_{\sigma \in S_{n}} q^{\crs(\sigma)} -\sum_{\sigma \in S_{n},\sigma(n)= n} q^{\crs(\sigma)}$ and $\displaystyle \sum_{\sigma \in S_{n},\sigma(n)= n} q^{\crs(\sigma)}=F_{n-1}(q)$, we obtain the following identity
\begin{eqnarray*}
	\sum_{\sigma \in S_{n+1}^2} q^{\crs(\sigma)}&=&q\left( F_{n}(q) -F_{n-1}(q)\right) +F_{n-1}(q).
\end{eqnarray*}
which is equivalent with the desired one of theorem	\ref{thm23}.
\end{proof}
When $k\geq 3$, the question arises naturally about the possible expressions for $\sum_{\sigma \in S_{n}^k} q^{\crs(\sigma)}$ that may use some property of $\phi_k$ (or $\psi_k$). We verify by computer up to n=6 the following conjecture.
\begin{conjecture}
	For all integers $n$ and $k$ satisfying $1\leq k\leq n$, we have 
	$$\sum_{\sigma \in S_{n}^k} q^{\crs(\sigma)}=\sum_{\sigma \in S_{n}^{n+1-k}} q^{\crs(\sigma)}.$$
\end{conjecture}

Now, let us look at how the restrictions of $\pi_1$ and $\psi_1$ may help us to get nice relationships between the polynomial distributions of the number of crossings over the set of 231 and 312-avoiding permutations. A result that may be useful in future research.
	
	\begin{theorem} \label{thm24}
		We have $F(312;q,z)=\displaystyle \frac{1}{1-zF(231;q,z)}$.
	\end{theorem}
	\begin{proof}
		Know that the restriction of $\phi_1$ (or $\psi_1$) on $S_n(231)$ is  a  bijection between $S_n(231)$ to $S_{n+1}^1(312)$ that preserves the statistic $\crs$. Consequently, we have 
		\begin{equation*}
		\sum_{\sigma \in S_{n+1}^1(312)} q^{\crs(\sigma)} =\sum_{\sigma \in S_{n}(231)} q^{\crs(\sigma)}.
		\end{equation*}
		
 Know also that each $\sigma \in S_n^j(312)$ can be written as the direct sum $\sigma_1\oplus\sigma_2:=\sigma_1.\sigma_2^{+j}$ of $\sigma_1 \in S_{j}^1$ and $\sigma_2\in S_{n-j}(312)$  where $\alpha^{+x}$ denote the obtained permutation from $\alpha$ by adding $x$ each of its letter, for any permutation $\alpha$ and integer $x\geq 1$ .  
 Consequently, the product $S_{j-1}(231)\times S_{n-j}(312)$ is in bijection with  the set $S_{n}^{j}(312)$  by the simple way: $(\alpha, \beta)\longmapsto \alpha^{-(j,1)}\oplus \beta$.  Hence, in terms of generating function, we get 
		\begin{eqnarray*}
			\sum_{\sigma \in S_{n}^{j}(312)} q^{\crs(\sigma)}&=&\sum_{\alpha \in S_{j-1}(231)} q^{\crs(\alpha)} \times \sum_{\beta \in S_{n-j}(312)} q^{\crs(\beta)}.
		\end{eqnarray*}
		This implies that
		\begin{eqnarray*}
			\sum_{\sigma \in S_{n}(312)} q^{\crs(\sigma)}&=& \sum_{j=0}^{n-1} \sum_{\alpha \in S_{j}(231)} q^{\crs(\alpha)} \times \sum_{\beta \in S_{n-1-j}(312)} q^{\crs(\beta)}.
		\end{eqnarray*}
		Thus, we get from this last identity the following functional equation 
		$$F(312;q,z) =1+z\ F(231;q,z)\times F(312;q,z).$$
		Solving it for $F(312;q,z)$, we get the proof of theorem \ref{thm24}.
	\end{proof}
	\section{Combinatorial interpretations of $\binom{n}{k}$}\label{sec3}
	In this section,  we enumerate the set of (123,132) and (123,213)-avoiding permutations according to the number of crossings and provide its connection to the binomial coefficients. That naturally implies the proof of theorem \ref{main1}.  Our method is based on exploitation of the fine structure of $S_n(123,132)$.  Then, we compute the closed formula for $\sum_{\sigma \in S_n(123,132)}q^{\crs(\sigma)}$ and extract its coefficients.  Since the only possible positions for 1 in  a permutation $\sigma \in S_n(123,132)$ are $n-1$ or $n$,  we can consider the partition $S_n(123,132)=S_n^1(123,132)\cup S_n^{2}(123,132)$.

	\begin{theorem} \label{thm31} Let $\T$ be one of the pairs $\{123,132\}$ and $\{123,213\}$. For all integer $n\geq 1$, we have $$\sum_{\sigma \in S_{n}(T)}q^{\crs(\sigma)} = \frac{(1+q)^{n-1}-1+q}{q}.$$
	\end{theorem}
	\begin{proof} Let us fix $\T=\{123,132\}$. Firstly, using the above described partition, we have
	\begin{equation}\label{eq31}
		\sum_{\sigma \in  S_{n}(\T)}q^{\crs(\sigma)}=
		\sum_{\sigma \in  S_{n}^{1}(\T)}q^{\crs(\sigma)}+ 
		\sum_{\sigma \in  S_{n}^{2}(\T)}q^{\crs(\sigma)}.
	\end{equation} 
Secondly, we can think of any restricted version of theorem \ref{thm23}. Since the restriction of $\phi_1$ on $S_{n-1}(\T)$, is a bijection from $S_{n-1}(\T)$ to $S_{n}^{1}(\T)$,  we get in terms of generating function 
\begin{equation}\label{eq32}
\sum_{\sigma \in S_{n}^{ 1}(\T)}q^{\crs(\sigma)}=F_{n-1}(\T;q).
\end{equation}
Similarly, the restriction of $\phi_2$ on $S_{n-1}(\T)$ satisfies
\begin{equation*}		
\crs(\phi_2(\sigma))=	
\begin{cases}
\crs(\sigma)=0  &    \text{ if } \ \ \sigma(n-1)=n-1 \\
\crs(\sigma)+1  &    \text{ if } \ \ \sigma(n-1)<n-1
\end{cases} 
\end{equation*} 
since the only $\sigma \in S_{n-1}(\T)$ satisfying $\sigma(n-1)=n-1$ is $\sigma=(n-2)\ldots 21(n-1)$.	Consequently, we get	
\begin{equation}\label{eq33}
\sum_{\sigma \in S_{n}^{2}(\T)}q^{\crs(\sigma)}=qF_{n-1}(T;q)+1-q.
\end{equation}		
Thirdly, when we observe and combine the three relations (\ref{eq31}), (\ref{eq32}) and (\ref{eq33}), we get the following recurrence
		\begin{equation*}F_{n}(\T;q)=(1+q)F_{n-1}(\T;q)+1-q
		\end{equation*}
Finally, when we solve this	recurrence with initial condition $F_{1}(\T;q)=1$, we obtain the closed formula  $F_{n}(\T;q)=\frac{(1+q)^{n-1}-1+q}{q}$. So, to complete the proof of theorem \ref{thm31} we use the fact that $\{123,213\}=_{\crs}\{123,132\}$ since $\{123,213\}=\rci(\{123,132\})$.
	\end{proof}
	When we extract the coefficients of $F_{n}(T;q)$, we obtain the following enumeration result.
	\begin{corollary} Let $\T$ be one of the pairs $\{123,132\}$ and $\{123,213\}$.
		For all integers $n\geq 1$ and $k\geq 0$, we have 
		\begin{equation*}
		|\{\sigma \in S_n(\T)/\crs(\sigma)=k\}|=\delta_{k,0}+\binom{n-1}{k+1}.
		\end{equation*} 
	\end{corollary} 

\hspace{-0.6cm}\textbf{Proof of theorem \ref{main1}.} Let $\T=\{123,132\}$  and $\T'=\{123,213\}$. Know first that, from (\ref{rel21}), we have
$$\sum_{\sigma \in S_{n}^2(\T)}q^{\crs(\sigma)}=\sum_{\sigma \in S_{n,2}(\T')}q^{\crs(\sigma)}.$$ 
So, deriving from the proof of the previous theorem, we easily obtain that of theorem \ref{main1}. Indeed, we can have 
\begin{eqnarray*}
\sum_{\sigma \in S_{n}^2(\T)}q^{\crs(\sigma)} &= &F_{n}(\T;q)-F_{n-1}(\T;q)\\
&=& \frac{(1+q)^{n-1}-1+q}{q}-\frac{(1+q)^{n-2}-1+q}{q}\\
&=& (1+q)^{n-2}.
\end{eqnarray*} 
This complete the proof of theorem  \ref{main1}.
	
	\begin{corollary} Let $\T=\{123,132\}$ and $\T'=\{123,213\}$. For all integer $n\geq 2$, we have
		\begin{equation*}
		|\{\sigma \in S_n^2(\T)/\crs(\sigma)=k \}|=|\{\sigma \in S_{n,2}(\T')/\crs(\sigma)=k \}|=\binom{n-2}{k}.
		\end{equation*}
	\end{corollary}

	\section{Combinatorial interpretations of $(R_n^k(q))$}\label{sec4}	
	In this section, we will show how the q-tableau $(R_n^k(q))$ defined in section \ref{sec1} counts the set of T-avoiding permutations according to the number of crossings where $\T$ is one of the two patterns $ \{213,312\},\{132,312\}, \{213,312\}$ and $\{213,312\}$. These patterns are linked by the following relations $\{132,312\}=\rci(\{213,312\})$,  $\{132,231\}=\rci(\{213,231\})$ and $\{132,231\}=\{213,312\}^{-1}$ or $\{213,312\}^{\rc}$.  So, we choose to consider only the pattern $\{213,312\}$.

	So, let us fix $\T=\{213,312\}$. Since each \T-voiding permutation stars or ends with 1, we have $S_n(\T)=S_n^n(\T)\cup S_n^1(\T)$. Consequently, we have
	\begin{equation}\label{eq41}
	\sum_{\sigma \in S_n(\T)}q^{\crs(\sigma )}=\sum_{\sigma \in S_n^{n}(\T)}q^{\crs(\sigma )}+\sum_{\sigma \in S_n^{1}(\T)}q^{\crs(\sigma )}
	\end{equation}
	The  following proposition is obvious.
	\begin{proposition}\label{prop41}
		For all integer $n\geq 1$, we have $\displaystyle \sum_{\sigma \in S_n^{n}(\T)}q^{\crs(\sigma )}=F_{n-1}(\T;q).$
	\end{proposition}
	
	Before we compute $\sum_{\sigma \in S_n^{1}(\T)}q^{\crs(\sigma )}$, we will denote by $F_{n}^{k}(\T;q)=\sum_{\sigma \in S_{n}^{[k]}(\T)}q^{\crs(\sigma)} $ for $1\leq  k\leq n$. Observe that, we have $\sum_{\sigma \in S_n^{1}(\T)}q^{\crs(\sigma )}=F_{n}^{1}(\T;q)$ since $S_{n}^{1}=S_{n}^{[1]}$. $F_{n}^{1}(\T;q)$ can be computed from a recursively formula for $F_{n}^{k}(\T;q)$ that we will prove later (see proposition \ref{prop42}). For that, we will denote by  $A_j(\sigma)=\{i<j/\sigma(i)\geq j\}$ , $B_j(\sigma)=\{i+1<j/\sigma(i)\leq i \text{ and  } i+1\leq \sigma^{-1}(i+1)\}$ and $C_j(\sigma)=\{(i,k)/i<k<\sigma(i)=k+1<\sigma(k) \text{ and } k+1\leq j\}$ for any given permutation $\sigma \in S_n$ and integer $j\in [n]$.
	
	\begin{lemma}\label{lem41}
		For all $\sigma \in S_n$, we have $$\crs(\sigma^{(1,j)})=\crs(\sigma)+|A_j(\sigma)|+|B_j(\sigma)|-|C_j(\sigma)| \text{ for } 1\leq j\leq n.$$
	\end{lemma}
	\begin{proof}
		By definition, we have $\sigma^{(1,j)}(1)=j$ and  $\sigma^{(1,j)}(1+i)=\begin{cases}
		\sigma(i)+1 & \text{ if } \sigma(i)\leq j\\
		\sigma(i) & \text{ if } \sigma(i)< j
		\end{cases}.$
		The following facts comes from this definition.
		\begin{enumerate}
			\item  if $i\in A_j(\sigma)$ then  $(1,1+i)$ becomes an upper crossings of $\sigma^{(1,j)}$.	
			\item  If $i\in B_j(\sigma)$ then  $1+i$ becomes a lower transient, then a crossing for $\sigma^{(1,j)}$.	
			\item If $(i,k) \in C_j(\sigma)$  then $i+1<\sigma^{(1,j)}(1+i)=1+k<\sigma^{(1,j)}(1+k)$. So, the pair $(1+i,1+k)$ becomes an upper crossing of $\sigma^{(1,j)}$ which does not counted as crossing of $\sigma^{(1,j)}$.
			\item It is not difficult to see that  $(i,k)$ is a crossing of $\sigma$ which does not in $C_j(\sigma)$ if only if $(i+1,k+1)$ is one of $\sigma^{(1,j)}$.
		\end{enumerate} 
		We obtain lemma \ref{lem41} from these four properties. 
	\end{proof}
	\begin{corollary}\label{cor41}
		If $\sigma \in S_{n}^{[k]}(\T)$, we have $\crs(\sigma^{(1,k+1)})=\min\{k-1,n-k\}+\crs(\sigma)$.
	\end{corollary}
	\begin{proof}
		Since every $\sigma \in S_{n}^{[k]}(\T)$ can be written as $\sigma =\pi|k\ldots 2 1$ such that $\pi$ is a \T-avoiding permutation  of $\{k+1,\ldots,n-1,n\}$, we have  $\sigma(i)\geq k+1$ for all $1\leq i\leq n-k$.
		Consequently, we  have $B_{k+1}(\sigma)=C_{k+1}(\sigma)=\emptyset$,   $A_{k+1}(\sigma)=\{1,2,\ldots, k-1\}$ if $k\leq \frac{n}{2}$ and $A_{k+1}(\sigma)=\{1,2,\ldots, n-k\}$ if $k> \frac{n}{2}$. In other word, we have
		$$|B_{k+1}(\sigma)|=|C_{k+1}(\sigma)|=0  \text{ and  }
		|A_{k+1}(\sigma)|=\begin{cases}
		k-1 & \text{ if }\ \ \ k\leq \frac{n}{2}\\
		n-1-k & \text{ if } \ \ \ k> \frac{n}{2}
		\end{cases} ~~\text{ for all } k.$$
		That means $|A_{k+1}(\sigma)|=\min\{k-1,n-1-k\}$.
		So, using lemma \ref{lem41}, we obtain our corollary.
	\end{proof}
	Let us define the \textit{skew sum} of two given permutations $\alpha$ and $\beta$ as $\alpha \ominus \beta:=\alpha^{+|\beta|}.\beta$. For example, we have $312\ominus 1342=7561342$. Using this last corollary, we can prove the following proposition which is the main tool to prove theorem \ref{main2}.
\begin{proposition}\label{prop42} Let $\T=\{213,312\}$.
	For all integers $n,k\geq 1$, we have 
	\begin{equation*}\label{eq:prop42}
	F_{n}^k(\T;q)=q^{\min\{k-1,n-1-k\}}F_{n-1}^k(\T;q)+F_{n}^{k+1}(\T;q).
	\end{equation*}
\end{proposition}
\begin{proof}
	Let $\T=\{213,312\}$. Firstly, since 
	\begin{equation*}
	S_{n}^{[k]}(\T)=\{\sigma\ominus k\ldots 21/\sigma \in S_{n-k}^{n-k}(\T)\}\cup \{\sigma\ominus k\ldots 21/\sigma \in S_{n-k}^1(\T) \}
	\end{equation*}
	then we have  $
	S_{n}^{[k]}(\T)= S_{n,*}^{[k]}(\T)\cup S_{n}^{[k+1]}(\T)$ where $S_{n,*}^{[k]}(\T)=\{\sigma\ominus k\ldots 21/\sigma \in S_{n-k}^{n-k}(\T)\}$.
	In terms of generating function, we get 
	\begin{equation}\label{eqn3x}
	\sum_{\sigma \in S_n^{[k]}(\T)}q^{\crs(\sigma)}= \sum_{\sigma \in S_{n,*}^{[k]}(\T)}q^{\crs(\sigma)}+ \sum_{\sigma \in S_{n}^{[k+1]}(\T)}q^{\crs(\sigma)}.
	\end{equation} 
	Secondly, the map from  $S_{n-1}^{[k]}(213,312)$ to $S_{n,*}^{[k]}(213,312)$  which sends $\sigma$ to $\sigma^{(1,k+1)}$ is well defined and bijective. Using corollary \ref{cor41}, we obtain
	\begin{equation}\label{eqn3y}
	\sum_{\sigma \in S_{n,*}^{[k]}(\T)}q^{\crs(\sigma)}=q^{\min\{k-1,n-1-k\}}\times F_{n-1}^k(\T;q) .
	\end{equation}
	So, combining identities (\ref{eqn3x}) and (\ref{eqn3y}), we complete the proof of proposition \ref{prop42}.
\end{proof}

\hspace{-0.6cm}\textbf{Proof of theorem \ref{main2}}. Let $\T=\{213,312\}$. 
It is obvious that $F_{n}^n(\T;q)=F_{n}^{n-1}(\T;q)=1$ since $S_{n}^{[n]}(\T)=S_{n}^{[n-1]}(\T)=\{n\ldots 21\}$. Combining propositions \ref{prop41} and \ref{prop42}, then applying relation (\ref{eq41}), we get the following ones
	\begin{equation} \label{eq42}
	\begin{cases}
	F_{n}^{n-1}(\T;q)&=F_{n}^{n-1}(\T;q)=1.\\
	F_{n}^{k}(\T;q)&=q^{\min\{k-1,n-1-k\}} F_{n}^k(\T;q)+F_{n}^{k+1}(\T;q) \text{ for all $k\geq 1$}.\\
	F_{n}(\T;q)&=F_{n-1}(\T;q)+F_{n}^{1}(\T;q).
	\end{cases}
	\end{equation}
	We can show easily by induction on $n$ that the two relationships  (\ref{q-tableau}) and (\ref{eq42})  are the same. So, we have $F_{n}(213,312;q)=R_{n}^0(q)$ and $F_{n}^{k}(\T;q)=R_{n}^k(q)$ for $k\geq 1$. Using the fact that $\{132,312\}=\rci(\{213,312\})$, we complete the proof of theorem \ref{main2}.
	
	\begin{corollary} 
		The number of (213,312)- or (132,312)-avoiding noncrossing permutations of  $[n]$ is the n-th central polygonal number $\binom{n}{2}+1$.
	\end{corollary}
	\begin{proof} It is easy to prove it by induction on $n$ that $F_{n}(213,312;0)=R_n^0(0)=\binom{n}{2}+1$. 
	\end{proof}
	\begin{theorem} \label{thm42}
		Let $\T$ be one of the pairs  $\{213,231\}$ and $\{132,231\}$.
		For all integer $n$,	we have 
		\begin{equation*}\label{eq3x}
		\sum_{\sigma \in S_n(\T)}q^{\crs(\sigma)}=R_{n+1}^1(q) 
		\end{equation*}
	\end{theorem}
	\begin{proof}
		The restriction of $\phi_1$ (or $\psi_1$) on  $S_{n}(213,231)$ is a bijection from  $S_{n}(213,231)$ to $S_{n+1}^1(213,312)$. Translating this fact into generating function, we have
		\begin{equation*}
		\sum_{\sigma \in S_n(213,231)}q^{\crs(\sigma)}=\sum_{\sigma \in S_{n+1}^1(213,312)}q^{\crs(\sigma)}=R_{n+1}^1(q).
		\end{equation*}
		We complete the proof of theorem \ref{thm42} using the fact that $\{132,231\}=_{\crs}\{213,231\}$ since $\{132,231\}=\rci(\{213,231\})$.
	\end{proof}
\section{Excedances and crossings over $S_n(321,231)$}

Let us first denote by $B_{n,k}:=\{\sigma \in S_n/\max\{i-\sigma(i)\}\leq k\}$ the subset of $S_n$  known as the set of permutations with maximum drop less than $k$ and $B_{n,k}(y,q)=\sum_{\sigma \in B_{n,k}}y^{\des(\sigma)}q^{\inv(\sigma)}$ where 
$\des(\sigma):=|\{i/\sigma(i)>\sigma(i+1)\}|$ and 
$\inv(\sigma):=|\{(i,j)/i<j,  \sigma(i)>\sigma(i)\}|$  are respectively the number of descents and inversions of $\sigma$. Explicit formula for $B_{n,k}(q)$ was given by Chung et al. in \cite{Chung} for $k\geq 1$. 
Here, we are interested on the case $k=1$ throughout  the following proposition.
\begin{proposition}\label{prop51}
	For all integer $n$, we have  $S_n(321,231)=B_{n,1}$.
\end{proposition}
\begin{proof}
	Know that each  $(321,231)$-avoiding permutation $\sigma$ can be written as $\sigma=1\oplus\pi$ with $\pi \in S_{n-1}(321,231)$ or $\sigma=j12 \ldots (j-1)\oplus\pi$ with $\pi \in S_{n-j}(321,231)$ where $1<j\leq n$.
	So, it is obvious  that we have  $S_n(321,231)\subset B_{n,1}$. 
	
	Let $\sigma \in B_{n,1}$. The following properties hold.
	\begin{itemize}
		\item[-]  If $\sigma(1)=1$, we have $\sigma=1\oplus \pi$ such that $\pi \in B_{n-1,1}$.
		\item[-] If $\sigma(1)=j\neq 1$, we have $\sigma=j12\ldots(j-1)\oplus \pi$ such that $\pi \in R_{n-j,1}$. Indeed, if $\sigma(2)\neq 1$ then there exists $i\geq 3$ such that $\sigma(i)=1$ and $i-1>2$. Contradiction since $\max\{i-\sigma(i)\}\leq 1$. By the same way, we can show that $\sigma(i)=i-1$ for all $i\in \{3,\ldots,j\}$. We get consequently $\sigma(1)\sigma(1)\ldots \sigma(j)=j12\ldots(j-1)$. Thus we have $\sigma=j12\ldots(j-1)\oplus \pi$ with $\pi \in B_{n-j,1}$. 
	\end{itemize} 
	Using these two properties, we can show easily by induction on $n$ that 	$B_{n,1}\subset S_n(321,231)$. So, we have $S_n(321,231)=B_{n,1}$.
\end{proof}
The following joint distribution of the number of descents and inversions was computed by Chung et al. in \cite{Chung}
\begin{equation}\label{eq:chung}
\sum_{\sigma\in  S(321,213)}y^{\des(\sigma)}q^{\inv(\sigma)}z^{|\sigma|}=\frac{1-qz}{1-(1+q)z-q(y-1)z^2}.
\end{equation}
The distribution of the number of inversions given by Dokos et al. \cite{Dokos} can be recovered from (\ref{eq:chung}).
\begin{equation}\label{eq:dokos}
\sum_{\sigma\in  S_n(321,231)}q^{\inv(\sigma)}=(1+q)^{n-1}.
\end{equation}

 Because of its connection to the results of  Dokos et al. and Chung et al., we will compute in the rest of this section the joint distribution of the statistics $\exc$ and $\crs$ over the set of (231,321)-avoiding permutations, where $\exc(\sigma):=|\{i/ \sigma(i)>\sigma(i+1)\}|$ is the number of excedances of $\sigma$. 
\begin{theorem}
	We have 
	\begin{eqnarray}\label{eq51}
	\sum_{\sigma \in S(231,321)} y^{\exc(\sigma)}q^{\crs(\sigma)} z^{|\sigma|} &  = &  \frac{1-qz}{1-(1+q)z-(y-q)z^2}.
	\end{eqnarray}
\end{theorem}
\begin{proof}
We will denote by $F_n(\T;y,q)=\sum_{\sigma \in S_n(\T)}y^{\exc(\sigma)}q^{\crs(\sigma)}$ the joint distribution of the two statistics $\exc$ and $\crs$ over the set $S_n(\T)$ for any subset of permutations $\T$. 	Let us fix $\T=\{321,231\}$. Let $\sigma$ be a $\T$-avoiding permutation. According to the structure of $\sigma$ described in the proof of proposition \ref{prop41}, we have
	\begin{enumerate}
		\item $(\exc(\sigma),\crs(\sigma))=(\exc(\pi),\crs(\pi))$ \ if \ $\sigma=1\oplus \pi$,
		\item $(\exc(\sigma),\crs(\sigma))=(1+\exc(\pi),j-2+\crs(\pi))$ if  $\sigma=j12 \ldots (j-1)\oplus\pi$ and $j\geq 2$.
	\end{enumerate}
	In terms of generating function, we obtain		
	\begin{eqnarray*}
		\sum_{1\oplus\pi=\sigma \in S_n(\T)}y^{\exc(\sigma)}q^{\crs(\sigma)}&=&F_{n-1}(\T;y,q) \text{ and  }\\
		\sum_{j12 \ldots (j-1)\oplus\pi=\sigma \in S_n(\T)}y^{\exc(\sigma)}q^{\crs(\sigma)}&=&yq^{j-2} F_{n-j}(\T;y,q) \text{ for all $j\geq 2$}.\\
	\end{eqnarray*}
	So we get 
	\begin{eqnarray*}
		F_{n}(\T;y,q) &  = &  	F_{n-1}(\T;y,q)+\sum_{j=2}^{n}yq^{j-2} F_{n-j}(\T;y,q)
	\end{eqnarray*}
	When we compute $F_{n}(\T;y,q)-qF_{n-1}(\T;y,q)$, we get
	$$F_{n}(\T;y,q)-qF_{n-1}(\T;y,q)=F_{n-1}(\T; y,q)+(y-q)F_{n-2}(\T;y,q)$$	
	Thus, we get the following recurrence relation
	\begin{equation}\label{eq:eq51x}
	F_{n}(\T;y,q)=(1+q)F_{n-1}(\T;y,q)+(y-q)F_{n-2}(\T;y,q)
	\end{equation}	
	If we denote by $F(\T;y,q,z):=1+\sum_{n\geq 1}F_{n}(\T;y,q)z^n$ the generating function of $F_{n}(\T;y,q)$, we obtain the following functional equation when we use (\ref{eq:eq51x}). 
	\begin{eqnarray*}
		F(\T;y,q,z)=1+z+(1+q)z(F(\T;y,q,z)-1)+(y-q)z^2F(\T;y,q,z).
	\end{eqnarray*}
	Solving it for $F(\T;y,q,z)$, we finally get the following identity
	\begin{eqnarray*}
	F(\T;y,q,z) &  = &  \frac{1-qz}{1-(1+q)z-(y-q)z^2}.
	\end{eqnarray*}	
 which is equivalent with (\ref{eq51}).
\end{proof}		
Observe that, since $\inv(\sigma)=\exc(\sigma)+\crs(\sigma)$ for all permutation $\sigma \in S_n(321,231)$ (see \cite{ARandr,Rakot} ), we then recover from (\ref{eq51}) the result of \cite{Dokos, Chung} about the distribution of the number of inversions over the set $S_n(321,231)$. We also get an unexpected result which is a refinement of the following identity that was first proved bijectively by Foata in \cite{Foata}
	\begin{equation}\label{eqfoata}
\sum_{\sigma\in S_n}y^{\des(\sigma)}=\sum_{\sigma \in S_n}y^{\exc(\sigma)} \text{ for all integer $n$}.
\end{equation}
The two statistics $\exc$ and $\des$ are known as \textit{eulerian statistics}.  Here, we get the following corollary which may be a refinement of (\ref{eqfoata}).
\begin{corollary}
	For all integer $n$, we have
	\begin{equation*}
	\sum_{\sigma\in S_n(231,321)}y^{\des(\sigma)}=\sum_{\sigma \in S_n(231,321)}y^{\exc(\sigma)}=\sum_{k\geq 0}\binom{n}{2k}y^k.
	\end{equation*}
\end{corollary}
\begin{proof}
	When we set $q=1$ in (\ref{eq:chung}) and $q=1$ in (\ref{eq51}), we obtain the same expression for $\sum_{\sigma\in  S(321,231)}q^{\des(\sigma)}z^{|\sigma|}$ and $\sum_{\sigma\in  S(321,231)}q^{\exc(\sigma)}z^{|\sigma|}$.
\end{proof}
\hspace{-0.7cm} We end this paper with the following enumeration result of (321,231)-avoiding noncrossing permutations.
\begin{corollary} The number of (231,321)-avoiding noncrossing permutations of $[n]$ is the n-th Fibonacci number $F_n$.	
\end{corollary}	
\begin{proof}
	If we set $q=0$ and $y=1$ in (\ref{eq51}), we obtain the generating function of the Fibonacci numbers. We also obtain the recurrence relation for the Fibonacci sequence $(F_n)$ when we use (\ref{eq:eq51x}). 
\end{proof}

\hspace{-0.58cm}\textbf{Acknowledgments}\\
We would like to thank Arthur Randrianarivony for helpful comments on this paper to get this version.

\end{document}